\pdfoutput=1
\RequirePackage{ifpdf}
\ifpdf % We~are running pdfTeX in pdf mode
\documentclass[pdftex]{sigma}
\else
\documentclass{sigma}
\fi

\numberwithin{equation}{section}

\newtheorem{Theorem}{Theorem}[section]
\newtheorem{Corollary}[Theorem]{Corollary}

\newtheorem{Proposition}[Theorem]{Proposition}
 { \theoremstyle{definition}
\newtheorem{Definition}[Theorem]{Definition}

\newtheorem{Remark}[Theorem]{Remark} }

\def\frak{\mathfrak}
\def\Bbb{\mathbb}
\def\Cal{\mathcal}
\let\phi\varphi
\newcommand{\x}{\times}
\renewcommand{\o}{\circ}

\newcommand{\al}{\alpha}

\newcommand{\ka}{\kappa}

\newcommand{\om}{\omega}

\newcommand{\ps}{\psi}

\newcommand{\si}{\sigma}

\newcommand{\Ga}{\Gamma}
\newcommand{\La}{\Lambda}

\newcommand{\Om}{\Omega}
\newcommand{\Si}{\Sigma}

\newcommand{\gr}{\operatorname{gr}}

\newcommand{\Ad}{\operatorname{Ad}}

\newcommand\xbbb[4]{%
\begin{picture}(66,12)\put(5,2){\line(1,0){16}}\put(25,2){\line(1,0){16}}\put(45,2){\line(1,0){16}}%
\put(3,2){\makebox(0,0){$\x$}}\put(23,2){\makebox(0,0){$\o$}}%
\put(43,2){\makebox(0,0){$\o$}}\put(63,2){\makebox(0,0){$\o$}}%
\put(3,9){\makebox(0,0){\scriptsize $#1$}}%
\put(23,9){\makebox(0,0){\scriptsize $#2$}}%
\put(43,9){\makebox(0,0){\scriptsize $#3$}}%
\put(63,9){\makebox(0,0){\scriptsize $#4$}} \end{picture}}

\newcommand\xxbb[4]{%
\begin{picture}(66,12)\put(5,2){\line(1,0){16}}\put(25,2){\line(1,0){16}}\put(45,2){\line(1,0){16}}%
\put(3,2){\makebox(0,0){$\x$}}\put(23,2){\makebox(0,0){$\x$}}%
\put(43,2){\makebox(0,0){$\o$}}\put(63,2){\makebox(0,0){$\o$}}%
\put(3,9){\makebox(0,0){\scriptsize $#1$}}%
\put(23,9){\makebox(0,0){\scriptsize $#2$}}%
\put(43,9){\makebox(0,0){\scriptsize $#3$}}%
\put(63,9){\makebox(0,0){\scriptsize $#4$}} \end{picture}}

\newcommand\wxxbb[4]{%
\begin{picture}(91,12)\put(10,2){\line(1,0){26}}\put(40,2){\line(1,0){26}}\put(70,2){\line(1,0){16}}%
\put(8,2){\makebox(0,0){$\x$}}\put(38,2){\makebox(0,0){$\x$}}%
\put(68,2){\makebox(0,0){$\o$}}\put(88,2){\makebox(0,0){$\o$}}%
\put(8,9){\makebox(0,0){\scriptsize $#1$}}%
\put(38,9){\makebox(0,0){\scriptsize $#2$}}%
\put(68,9){\makebox(0,0){\scriptsize $#3$}}%
\put(88,9){\makebox(0,0){\scriptsize $#4$}} \end{picture}}

\begin{document}

\allowdisplaybreaks

\newcommand{\arXivNumber}{2405.13614}

\renewcommand{\PaperNumber}{108}

\FirstPageHeading

\ShortArticleName{On Relative Tractor Bundles}

\ArticleName{On Relative Tractor Bundles}

\Author{Andreas \v{C}AP, Zhangwen GUO and Micha{\l} Andrzej WASILEWICZ}

\AuthorNameForHeading{A.~\v{C}ap, Z.~Guo and M.A.~Wasilewicz}

\Address{Faculty of Mathematics, University of Vienna, Oskar-Morgenstern-Platz 1, 1090 Wien, Austria}
\Email{\href{mailto:andreas.cap@univie.ac.at}{andreas.cap@univie.ac.at},
 \href{mailto:zhangwen.guo@univie.ac.at}{zhangwen.guo@univie.ac.at}, \href{mailto:michal.wasilewicz@univie.ac.at}{michal.wasilewicz@univie.ac.at}}

\ArticleDates{Received June 03, 2024, in final form November 18, 2024; Published online November 29, 2024}

\Abstract{This article contributes to the relative BGG-machinery for parabolic geometries. Starting from a relative tractor bundle, this machinery constructs a~sequence of differential operators that are naturally associated to the geometry in question. In many situations of interest, it is known that this sequence provides a resolution of a sheaf that can locally be realized as a pullback from a local leaf space of a foliation that is naturally available in this situation. An explicit description of the latter sheaf was only available under much more restrictive assumptions. For any geometry which admits relative tractor bundles, we construct a large family of such bundles for which we obtain a simple, explicit description of the resolved sheaves under weak assumptions on the torsion of the geometry. In particular, we discuss the cases of Legendrean contact structures and of generalized path geometries, which are among the most important examples for which the relative BGG machinery is available. In both cases, we show that essentially all relative tractor bundles are obtained by our construction and our description of the resolved sheaves applies whenever the BGG sequence is a resolution.}


\Classification{58J10; 53C07; 53C15; 58J60; 58J70}

\section{Introduction}\label{1}
Tractor bundles and relative tractor bundles are a special class of vector bundles
that are naturally associated to manifolds endowed with a geometric structure from
the class of parabolic geometries. A major reason for their importance lies in the
role in the construction of Bernstein--Gelfand--Gelfand sequences (or BGG sequences)
and in the relative version of this construction, respectively. To explain this, we
have to review some background: The motivation for these constructions comes from
algebraic results in the realm of semisimple Lie algebras. The starting point was
\cite{BGG}, in which of I.N.~Bern{\v{s}}te{\u\i}n, I.M.~Gel'fand and S.I.~Gel'fand constructed a~resolution of any finite-dimensional
representation of a complex semisimple Lie algebra $\frak g$ by homomorphisms of
Verma modules. This was generalized by J.~Lepowsky to resolutions by homomorphisms of
generalized Verma modules for a given parabolic subalgebra $\frak p\subset\frak g$,
see \cite{Lepowsky}.

The connection to geometry comes from a duality connecting homomorphisms between
generalized Verma modules for $(\frak g,\frak p)$ to invariant differential operators
acting on sections of homogeneous vector bundles over the generalized flag variety
$G/P$ (for an appropriate Lie group~$G$ with Lie algebra $\frak g$ and subgroup
$P\subset G$ corresponding to $\frak p$) coming from irreducible representations of
$P$. The space $G/P$ in turn is the homogeneous model for parabolic geometries of
type $(G,P)$, a concept that extends to a general (real or complex) semisimple Lie
group $G$ and parabolic subgroup $P\subset G$. The parabolic geometries are defined
(uniformly) as Cartan geometries of type $(G,P)$ but there are general results
showing that they are equivalent descriptions of simpler underlying structures,
whose explicit descriptions are very diverse, see \cite{book}. Among these underlying
structures there are well-known examples like conformal structures and
Levi-non-degenerate CR structures of hypersurface type that are intensively studied
using a variety of other methods. In particular, understanding differential operators
intrinsic to such geometries has received a lot of interest.

It turned out that things can not only be moved to the geometric side, but the
geometric version of the construction then applies to arbitrary curved geometries of
type $(G,P)$. Building on \cite{Baston}, the first general geometric construction was
given in \cite{CSS-BGG}, with improvements in \cite{Calderbank-Diemer} and later in
\cite{Rel-BGG2}. The basic input for the construction is a so-called \textit{tractor
 bundle} that is induced by a representation $\Bbb V$ of the group $G$ (which is
restricted to $P$). On the homogeneous model~$G/P$, any such bundle is canonically
trivialized and hence endowed with a canonical linear connection. In curved cases,
tractor bundles are non-trivial in general, but they still inherit canonical
linear connections, whose curvature equivalently encodes the curvature of the Cartan
connection. Coupling these \textit{tractor connections} to the exterior derivative,
one obtains a twisted de Rham sequence.\looseness=-1

The BGG construction then compresses the operators in this sequence to a sequence of
higher order operators acting on sections of bundles associated to completely
reducible representations of $P$, which form the BGG sequence. The latter bundles are
induced by Lie algebra homology spaces and Kostant's version of the Bott--Borel--Weyl
theorem \cite{Kostant} implies that they correspond to the Verma modules showing up
in Lepowsky's construction. On locally flat geometries, the de~Rham sequence is a
fine resolution of the sheaf of local parallel sections of the tractor bundle, and it
turns out that the same holds for the induced BGG sequence, so in particular both
sequences compute the same cohomology. In the curved case, there still is a close
relation between the operators in the BGG sequences and the covariant exterior
derivative, which has been exploited very successfully in many applications.

In the setting of the homogeneous models, a relative version of the BGG construction
was already used in the book \cite{BEastwood} on the Penrose transform. Here one
considers two nested parabolic subgroups $Q\subset P\subset G$, which immediately
leads to a~projection $G/Q\to G/P$. It turns out that the fibers $P/Q$ of this
projection again are generalized flag varieties (of a smaller group) and one then
uses a ``BGG resolution along the fibers''. A geometric version of this relative
theory was obtained in \cite{Rel-BGG2} building on algebraic background from
\cite{Rel-BGG1}. Starting from a geometry of type~$(G,Q)$ on a manifold $M$, the
additional parabolic subgroup $P$ gives rise to a~natural distribution~${T_\rho
M\subset TM}$, called the \textit{relative tangent bundle}. The second ingredient one
needs is a~\textit{relative tractor bundle} which is induced by a completely
reducible representation of the group~$P$ and the theory of such representations is
well understood. It turns out that any such bundle inherits a canonical partial
connection \big(which allows for differentiation in directions in~$T_\rho M$ only\big),
called the \textit{relative tractor connection}. With some technical subtleties \big(in
particular in the case that $T_\rho M$ is not involutive\big) an analog of the geometric
BGG construction discussed above then produces operators between bundles associated
to completely reducible representations~of~$Q$.

For appropriate initial representations, the relative BGG construction produces
operators that cannot be obtained by the original BGG construction, in particular in
cases of so-called singular infinitesimal character. In other cases, it provides
alternative constructions of BGG operators. A very important feature is that the
relative construction provides complexes and resolutions in many non-flat cases. On
the one hand, this needs involutivity of the distribution~$T_\rho M$, which can be
easily characterized in terms of the curvature and the harmonic curvature of the
original geometry of type $(G,Q)$. Under slightly stronger curvature conditions one
obtains a fine resolution of a sheaf that locally is a pull-back of a sheaf on a
local leaf space for the distribution $T_\rho M$. An issue that remained open in
\cite{Rel-BGG2} is how to concretely identify this sheaf, even in the case that $M$
is globally the total space of a bundle over some base space $N$ with vertical
subbundle $T_\rho M$.

It is exactly this point that we address in the current paper. Given
 $Q\subset P\subset G$ as above, a~geometry of type $(G,Q)$ on $M$ canonically
 determines natural subbundles $T^{i'}_PM\subset TM$, see Section~\ref{2.3}. The
 conditions on the torsion of the geometry referred to below are on its behavior
 with respect to these subbundles. We show in each such situation, there is a large
 class of relative tractor bundles which, assuming involutivity of $T_\rho M$ and a
 weak condition on the torsion, can locally be explicitly identified with pullbacks
 of vector bundles on appropriate local leaf spaces for $T_\rho M$. This is proved
 in part (1) of Theorem~\ref{thm3.2.2}. Part (2) of this theorem shows that under a
 slightly stronger condition on the torsion, the sections of these relative tractor
 bundles, which are parallel for the relative tractor connection, locally are exactly the pullbacks of sections of the bundles on local leaf spaces.
These are exactly the sheaves that, assuming appropriate conditions on the curvature,
are resolved by the corresponding relative BGG sequence. Moreover, we can explicitly
relate the relative tractor connections on this bundle to well-known operations, like
the adjoint tractor connection and \big(in the case that $T_\rho M$ is involutive\big) to the
Bott connection associated to the induced foliation, see Theorem
 \ref{thm3.2.1}. These results originally arose during the work on the PhD theses
of the second and third author under the direction of the first author for two
specific examples of structures.

It should be remarked at this point that the relative theory makes sense only if $Q$
does not correspond to a maximal parabolic subalgebra of $\frak g$, since otherwise
there is no intermediate parabolic subgroup $P$. Consequently, the theory does not
apply to examples like conformal structures or CR structures of hypersurface
type. Still, there are at least two important examples of parabolic geometries for
which relative BGG sequences are available, namely Legendrean (or Lagrangean) contact
structures and generalized path geometries and these are the two examples referred to
above. Both these structures can exist only on manifolds of odd dimension. The former
consists of a contact structure together with a chosen decomposition of the contact
subbundle into the direct sum of two Legendrean subbundles. These structures are
closely similar to CR structures and have been studied intensively during the last
years and successfully applied to classification problems in CR geometry, see, e.g.,
\cite{Doubrov-Medvedev-The,Doubrov-Merker-The} and \cite{MFMZ}. For this
geometry, our construction essentially provides all relative tractor
bundles. Generalized path geometries are also defined by a configuration of
distributions, see Section~\ref{3.5} below. Their importance comes from the fact that they
provide an equivalent encoding of systems of second order ODEs, see \cite{Fels}. For
generalized path geometries, we obtain essentially all relative tractor bundles for
one of the possible intermediate parabolic subgroups and a large subclass for the
other.

Let us briefly outline the contents of the paper. In Section~\ref{2}, we first collect the
necessary setup for nested parabolics. As a new ingredient compared to the discussion
in \cite{Rel-BGG1,Rel-BGG2}, we focus on the natural bigrading of the Lie algebra
$\frak g$ induced by the nested parabolic subalgebras $\frak q\subset\frak
p\subset\frak g$, which is very convenient for our purposes. Considering an
appropriate filtration induced by this bigrading, one quickly arrives at a class of
representations that induce basic relative tractor bundles, which can be obtained as
subquotients of the adjoint tractor bundle of the geometry.\looseness=-1

In Section~\ref{3}, we start by showing that the relative tractor connections on the
bundles in our class are induced by the adjoint tractor connection. Next, we derive an
explicit formula for the relative tractor connections in terms of the Lie bracket of
vector fields and the torsion of the geometry. The relation to the Lie bracket is
crucial for the proof of Theorem~\ref{thm3.2.2}, which is the main technical result
of our article. Under appropriate conditions on the torsion, we obtain a~filtration
of the tangent bundle of local leaf spaces and relate the relative tractor bundles we
have constructed to the components of the associated graded vector bundle to this
filtration. Under slightly stronger conditions, it is shown that sections, which are
parallel for the relative tractor connection, coincide with pullbacks of sections on
the local leaf space. This quickly leads to a large class of relative tractor bundles
for which similar results are available, see Section~\ref{3.3}.

In the last part of the article, we discuss the two examples which gave rise to the
developments of this article, namely Legendrean contact structures in Section~\ref{3.4}
and (generalized) path geometries in Section~\ref{3.5}. In both cases, the situation is
relatively simple in the sense that the filtrations of the tangent bundle of local
leaf spaces are trivial, so the special relative tractor bundles we construct are all
directly related to tensor bundles on local leaf spaces. For both examples of
structures, there are two possible intermediate parabolic subgroups between~$Q$ and~$G$.
We in particular prove that, up to isomorphism, our construction produces
essentially all relative tractor bundles for both intermediate groups in the
Legendrean contact case and for one of the intermediate groups in the case of
generalized path geometries.

\section{A basic class of relative tractor bundles}\label{2}

We start by briefly recalling the setup of two nested parabolic subalgebras $\frak
q\subset\frak p$ in a semisimple Lie algebra $\frak g$ with a compatible choice
$Q\subset P\subset G$ of groups as discussed in \cite{Rel-BGG1}. The focus here is on
a natural bigrading induced by the pair $(\frak q,\frak p)$.

\subsection{The bigrading determined by a pair of parabolics}\label{2.1}
There is a well-known correspondence between standard parabolic subalgebras of a real
or complex semisimple Lie algebra $\frak g$, gradings of $\frak g$, and (after an
appropriate choice of Cartan subalgebra) subsets of the set of simple roots
respectively simple restricted roots. By definition, any (restricted) root can be
uniquely written as a linear combination of simple (restricted) roots with integer
coefficients that are either all $\geq 0$ or all $\leq 0$. Given a subset $\Si$ of
simple (restricted) roots, one defines the $\Si$-height of a root $\al$ to be the sum
of the coefficients of all elements of $\Si$ in this expansion of $\al$. The grading
corresponding to $\Si$ is then defined by putting the Cartan subalgebra into degree
zero and giving (restricted) root spaces degree equal to the $\Si$-height of the
root. The compatibility of the (restricted) root decomposition with the Lie bracket
of $\frak g$, readily implies compatibility of this grading with the Lie bracket. In
particular, the sum of all grading components of non-positive degree is a subalgebra
of $\frak g$ and this is the parabolic subalgebra determined by $\Si$. See \cite[Sections
3.2.1 and 3.2.9]{book} for details.

In this language, two nested parabolic subalgebras $\frak q\subset\frak p\subset\frak
g$ correspond to subsets $\Si_{\frak q}$ and $\Si_{\frak p}$ such that $\Si_{\frak
 p}\subset\Si_{\frak q}$. This readily implies that for each positive (restricted)
root $\al$ the $\Si_{\frak p}$-height of $\al$ is $\leq$ its $\Si_{\frak
 q}$-height. Using this observation, we define the bigrading determined by $(\frak
q,\frak p)$.

\begin{Definition}\label{def2.1}
 Consider a real or complex semisimple Lie algebra $\frak g$ endowed with two nested
 standard parabolic subalgebras $\frak q\subset\frak p\subset\frak g$ corresponding
 to subsets $\Si_{\frak p}\subset\Si_{\frak q}$ of simple (restricted) roots. Then
 we define $\frak g_{(0,0)}\subset\frak g$ to be the sum of the Cartan subalgebra
 and all (restricted) root spaces of roots for which both the $\Si_{\frak p}$-height
 and the $\Si_{\frak q}$-height is zero. For integers~$i'$,~$i''$ which have the same
 sign and are not both zero, we define $\frak g_{(i',i'')}$ to be the direct sum of
 all the restricted root spaces for roots of $\Si_{\frak p}$-height $i'$ and
 $\Si_{\frak q}$-height $i'+i''$.
\end{Definition}

We remark that the set $\Si_{\frak q}\setminus\Si_{\frak p}$ of simple roots gives
rise to a second parabolic subalgebra $\tilde{\frak p}\subset\frak g$ such that
$\frak q=\frak p\cap\tilde{\frak p}$. The integers $i''$ in Definition \ref{def2.1}
correspond to the grading of $\frak g$ induced by that parabolic. Since $\tilde{\frak
 p}$ does not play a role in the geometric developments below, we avoid using it
explicitly.

By construction we have a decomposition $\frak g=\bigoplus_{i',i''}\frak g_{(i',i'')}$
and we start by clarifying some elementary properties. Recall that the nilradical
$\frak p_+$ of a parabolic $\frak p$ is the sum all positive grading components,
while the reductive Levi-factor $\frak p_0$ is the degree zero component of the
grading determined by $\frak p$.

\begin{Proposition}\label{prop2.1}\samepage
The decomposition $\frak g=\bigoplus_{i',i''}\frak g_{(i',i'')}$ induced by $\frak
q\subset\frak p\subset\frak g$ has the following properties:
\begin{itemize}\itemsep=0pt
\item[$(i)$] $\bigl[\frak g_{(i',i'')},\frak g_{(j',j'')}\bigr]\subset \frak g_{(i'+i'',j'+j'')}$, so one
obtains a bigrading of the Lie algebra $\frak g$. In particular, the bracket vanishes
on $\frak g_{(0,i'')}\x \frak g_{(i',0)}$ if $i'>0$ and $i''<0$ or $i'<0$ and $i''>0$.

\item[$(ii)$] $\frak p=\bigoplus_{i'\geq 0,i''}\frak g_{(i',i'')}$, $\frak
p_+=\bigoplus_{i'>0,i''}\frak g_{(i',i'')}$, and $\frak p_0=\bigoplus_{i''}\frak g_{(0,i'')}$.

\item[$(iii)$] $\frak q=\bigoplus_{i',i''\geq 0}\frak g_{(i',i'')}$ and
$\frak q_+=\bigoplus_{i'+i''>0} g_{(i',i'')}$, and $\frak q_0=\frak g_{(0,0)}$.
\end{itemize}
\end{Proposition}
\begin{proof}
 Let us denote by $\Si_{\frak p}\subset\Si_{\frak q}$ the subsets of simple
 (restricted) roots corresponding to the two parabolics. Then the first claim in (i)
 follows immediately from additivity of the $\Si$-height for any subset $\Si$. The
 second claim follows since we know that the bracket has values in
 $\frak g_{(i',i'')}$ which by construction is zero under our assumptions on $i'$ and
 $i''$. The claims in (ii) follow since the first index records the
 $\Si_{\frak p}$-height. The claims in (iii) follow the sum of the two indices
 is the $\Si_{\frak q}$-height and the only possibility to obtain $i'+i''=0$ is
 $i'=i''=0$ since both indices have to be either non-negative or non-positive.
\end{proof}

The decomposition we obtain is a refinement of the one in formula (2.1) of \cite{Rel-BGG1}.
In the notation used there, we get
$\frak p_-=\bigoplus_{i'<0,i''\leq 0}\frak g_{(i',i'')}$,
$\frak p_0\cap\frak q_-=\bigoplus_{i''<0}\frak g_{(0,i'')}$,
$\frak p_0\cap\frak q_+=\bigoplus_{i''>0}\frak g_{(0,i'')}$ for the components are not
determined in Proposition \ref{2.1}.

\subsection{The group level and invariant filtrations}\label{2.2}
It is usual in the theory of parabolic geometries, that gradings have to be viewed as
auxiliary objects, since they are not invariant under the actions of parabolic
subalgebras. To obtain objects that are invariant, one has to pass to associated
filtrations. To do this, we first recall from \cite{Rel-BGG1} how to appropriately
choose groups in the situation of nested parabolics. Given
$\frak q\subset\frak p\subset\frak g$ as above, we first fix a group $G$ with Lie
algebra $\frak g$ and choose a parabolic subgroup $P\subset G$ corresponding to
$\frak p$. It is well known that this means that $P$ lies between the normalizer
$N_G(\frak p)$ of $\frak p$ in $G$ and its connected component of the
identity. Finally, we choose a subgroup $Q\subset P$ corresponding to the Lie
subalgebra $\frak q$, which means that it lies between $N_P(\frak q)$ and its
connected component of the identity. Since $P$ and $Q$ are parabolic subgroups, it is
well known that the exponential map restricts to diffeomorphisms from $\frak p_+$
onto a closed normal subgroup $P_+\subset P$ and from $\frak q_+$ onto a closed
normal subgroup $Q_+\subset Q$. From Proposition \ref{prop2.1}, we know that
$\frak p_+\subset\frak q_+$ so $P_+\subset Q_+\subset Q$.

Parabolic geometries of type $(G,Q)$ come with principal $Q$-bundles, so via the
construction of associated vector bundles, any representation of $Q$ determines a
natural vector bundle on each manifold endowed with such a geometry. Observe that by
construction the subgroup $Q_+\subset Q$ is nilpotent and it is well known that the
quotient $Q/Q_+$ is reductive. Hence, representations of $Q$ are not easy to
understand in general. Special classes of natural bundles are obtained via
considering a restricted class of representations of $Q$, which often are easier to
understand. Classical examples of this situation (which are available for any
parabolic subgroup) are \textit{completely reducible bundles} that correspond to
direct sums of irreducible representations of $Q$ (on which $Q_+$ automatically acts
trivially) and \textit{tractor bundles} that correspond to restrictions to $Q$ of
representations of the group $G$. So in both these cases the relevant representation
theory is well understood.

In the more restrictive situation of two nested parabolics $Q\subset P\subset G$, two
more classes of natural bundles were introduced in
\cite[Definition 2.1]{Rel-BGG2}. These are \textit{relative natural bundles} which correspond to
representations of $Q$ on which $P_+$ acts trivially and \textit{relative tractor
 bundles} which correspond to restrictions to $Q$ of representations of $P$ on which
$P_+$ acts trivially. So in particular any completely reducible representation of $P$
gives rise to a relative tractor bundle on geometries of type $(G,Q)$. There is a
simple family of appropriate representations of $P$, which are just the filtration
components of the filtration of $\frak g$ determined by $\frak p$.

\begin{Definition}\label{def2.2}
 Consider a real or complex semisimple Lie algebra $\frak g$ endowed with the
 bigrading $\frak g=\bigoplus_{i',i''}\frak g_{(i',i'')}$ determined by two nested
 standard parabolic subalgebras $\frak q\subset\frak p\subset\frak g$. Then for each
 $i'$, we define $\frak g^{(i',*)}:=\bigoplus_{j'\geq i',j''}\frak g_{(j',j'')}$.
\end{Definition}

The main properties of these subspaces are easy to clarify.

\begin{Proposition}\label{prop2.2}
 Consider a real or complex semisimple Lie algebra $\frak g$ endowed with the
 bigrading $\frak g=\bigoplus_{i',i''}\frak g_{(i',i'')}$ determined by two nested
 standard parabolic subalgebras $\frak q\subset\frak p\subset\frak g$. Then for each
 integer $i'$ the subspace $\frak g^{(i',*)}\subset\frak g$ is invariant for the
 restriction of the adjoint action of $G$ to $P$. Moreover, $\bigl[\frak p_+,\frak
 g^{(i',*)}\bigr]\subset\frak g^{(i'+1,*)}$, so the subgroup $P_+$ acts trivially on the
 quotient $\Bbb V_{i'}:=\frak g^{(i',*)}/\frak g^{(i'+1,*)}$.
 For each $i''$ such that $\frak g_{(i',i'')}\neq 0$, let $\Bbb V^{i''}_{i'}\subset
 \Bbb V_{i'}$ be the image of the~subspace \smash{$\bigoplus_{j'\geq i',j''\geq i''}\frak
 g_{(j',j'')}\subset\frak g^{(i',*)}$} under the quotient projection. Then each of the
 spaces~\smash{$\Bbb V^{i''}_{i'}$} is $Q$-invariant and they define a decreasing filtration
 of $\Bbb V_{i'}$ in the sense that for $i''\leq \ell''$ we get~${\Bbb
 V^{i''}_{i'}\supset \Bbb V^{\ell''}_{i'}}$.
\end{Proposition}
\begin{proof}
 The grading of $\frak g$ induced by $\frak p$ induces a filtration, whose
 components by construction are exactly the space $\frak g^{(i',*)}$. It is well
 known this components can be obtained algebraically from~$\frak p$ as the
 nilradical, the elements of its derived series and annihilators under the Killing
 form, compare with \cite[Corollary 3.2.1]{book}. Hence, any element of $N_G(\frak
 p)$ normalizes each of the~filtration components which shows that each $\frak
 g^{(i',*)}$ is $P$-invariant. In particular, $\frak p_+=\frak g^{(1,*)}$ and the
 second claim follows from Proposition \ref{prop2.1}. Applying this argument to
 $Q$, we conclude that each of the spaces $\bigoplus_{j'\geq i',j''\geq i''}\frak
 g_{(i',j')}$ is $Q$-invariant and from this the second part follows readily.
\end{proof}

\subsection{Some basic relative tractor bundles}\label{2.3}
We can now obtain the distinguished relative tractor bundles by passing to associated
bundles. Recall that, given a parabolic geometry $(p\colon \Cal G\to M,\om)$, the
associated bundle induced by the restriction of the adjoint representation of $G$ to
$Q$ is the \textit{adjoint tractor bundle} $\Cal AM=\Cal G\x_Q\frak g$. By
Proposition \ref{prop2.2}, we see that for each $i'$, we get a natural subbundle
corresponding to the subspace $\frak g^{(i',*)}\subset\frak g$ which is $P$-invariant
and hence $Q$ invariant. We denote this subbundle by~$\Cal A^{(i',*)}M\subset\Cal
AM$. Of course we can then form the quotient $\Cal V_{i'}M:=\Cal A^{(i',*)}M/\Cal
A^{(i'+1,*)}M$ which is isomorphic to \smash{$\Cal G\x_Q\bigl(\frak g^{(i',*)}/\frak
 g^{(i'+1,*)}\bigr)=\Cal G\x_Q\Bbb V_{i'}$}. Hence, by Proposition \ref{prop2.2}, each
$\Bbb V_{i'}$ is a relative tractor bundle. Moreover, the $Q$-invariant subspaces
$\Bbb V_{i'}^{i''}\subset\Bbb V_{i'}$ give rise to a filtration of $\Cal V_{i'}M$ by
smooth subbundles $\Cal V_{i'}^{i''}M$, each of which is a relative natural bundle.

In what follows, we will mainly be interested in the case where $i'<0$. The relevance
of this condition is that by construction $\frak g^{(0,*)}=\frak p\supset\frak q$. It
is well known that, via the Cartan connection~$\om$, the associated bundle $\Cal
G\x_Q\frak g/\frak q$ is isomorphic to $TM$. Now the subspace $\frak p\subset\frak g$
is $P$-invariant and hence $Q$-invariant, so $\frak p/\frak q\subset\frak g/\frak q$
is a $Q$-invariant subspace, too. Hence, it defines a natural subbundle in $TM$, which
is called the \textit{relative tangent bundle} $T_\rho M$ in \cite{Rel-BGG2}. For
$i'<0$, we get in the same way a $Q$-invariant subspace $\frak g^{(i',*)}/\frak q$
which gives rise to a smooth subbundle of $TM$ that contains $T_\rho M$. We denote
this bundle by $T^{i'}_P M$ to distinguish it form the filtration component coming
from the initial parabolic geometry. By construction, we then can naturally identify
$\Cal V_{-1}M$ with $T^{-1}_PM/T_\rho M$ and, for $i'<-1$, $\Cal
V_{i'}M=T^{i'}_PM/T^{i'+1}_PM$, so we always obtain natural subquotients of the
tangent bundle. This will be crucial for the later developments. In what follows it
will be convenient to define $T^0_PM:=T_\rho M$, which in particular allows us to
uniformly write $\Cal V_{i'}M=T^{i'}_PM/T^{i'+1}_PM$ for any $i'\leq -1$.

A particularly important special case is that the relative tangent bundle $T_\rho
M\subset TM$ is an involutive distribution on $M$. In this case, we can find
\textit{local leaf spaces} for this distribution. More explicitly, for any $x\in M$
there is an open neighborhood $U\subset M$ of $x$ and a surjective submersion
$\ps\colon U\to N$ onto some smooth manifold $N$ such that, for any $y\in U$,
$\ker(T_y\ps)$ coincides with the fiber of $T_\rho M$ at $y$. In particular, this
implies that $T\ps$ induces an isomorphism $(TM/T_\rho
M)|_U\cong\psi^*TN$. Characterizing involutivity of $T_\rho M$ and the question of
which part of the given geometric structure on $M$ descends to local leaf spaces is a
key part of the theory of twistor spaces for parabolic geometries that was originally
developed in \cite{twistor}. Recall that the curvature of the Cartan connection $\om$
can be interpreted as a two-form $\ka\in\Om^2(M,\Cal AM)$ and since $TM$ naturally is
a quotient of $\Cal AM$, there is a projection $\tau\in\Om^2(M,TM)$ of $\ka$, called
the \textit{torsion} of the Cartan connection $\om$. By \cite[Proposition 4.6]{Rel-BGG2},
involutivity of $T_\rho M$ is equivalent to the fact that $\tau$
maps $T_\rho M\x T_\rho M$ to $T_\rho M$.

There is a strengthening of this condition that plays an important role in the theory
of relative BGG sequences, namely that $\ka$ vanishes on $T_\rho M\x T_\rho M$. (In
\cite{Rel-BGG2}, this is phrased as vanishing of the relative curvature.) By \cite[Theorem 4.11]{Rel-BGG2},
this implies that the relative BGG sequence determined by
any relative tractor bundle is a complex and a fine resolution of a sheaf on
$M$. This sheaf can be identified with the sheaf of those local sections of the
relative tractor bundle which are parallel for the relative tractor connection, which
will be discussed in Section~\ref{3.1} below. It is shown in \cite[Section 4.7]{Rel-BGG2}
that this implies that locally the resolved sheaf can be written as a pullback of a
sheaf on a local leaf space. In the special case of a so-called correspondence space
(which means that the geometry descends to the leaf space), that sheaf is identified
explicitly in \cite[Theorem 4.11]{Rel-BGG2}. In general, this sheaf is not
identified in \cite{Rel-BGG2}, we will present a solution to this question under some
assumptions below.

\section{Relative tractor connections and resolved sheaves}\label{3}

\subsection{Restrictions of the adjoint tractor connection}\label{3.1}
A crucial feature of tractor bundles is that they carry canonical linear connections
known as tractor connections. There is a uniform description of these tractor
connections, based on the so-called fundamental derivative. For a geometry $(p\colon \Cal
G\to M,\om)$ of type $(G,Q)$ consider any representation $\Bbb W$ of $Q$ and the
corresponding natural vector bundle $\Cal WM:=\Cal G\x_Q\Bbb W$. Then smooth sections
of $\Cal WM$ are in bijective correspondence with smooth functions $\Cal G\to\Bbb W$
which are $Q$-equivariant in the sense that $f(u\cdot g)=g^{-1}\cdot f(u)$ for any
$g\in Q$, where in the left-hand side we use the principal right action of $Q$ on
$\Cal G$ and in the right-hand side the given representation on $\mathbb{W}$. In this way,
sections of the adjoint tractor bundle $\Cal AM$ get identified with $Q$-equivariant
smooth functions $f\colon \Cal G\to\frak g$. But via the Cartan connection $\om$, smooth
functions $\Cal G\to\frak g$ are in bijective correspondence with vector fields on
$\Cal G$, and the equivariancy condition on $f$ is equivalent to the corresponding
vector field $\xi\in\frak X(\Cal G)$ being invariant under the principal right action
of $Q$, i.e., $(r^g)^*\xi=\xi$ for any $g\in Q$.

Differentiating an equivariant function with respect to an invariant vector field
produces an equivariant function, so we can view this as defining an operator
$D\colon \Gamma(\Cal AM)\x \Ga(\Cal WM)\to\Ga(\Cal WM)$ called the \textit{fundamental
 derivative}, see \cite[Section 1.5.8]{book}. This has strong invariance
properties and to emphasize the analogy to a covariant derivative, it is usually
written as~${(s,\si)\!\mapsto\! D_s\si}$.

Now for a tractor bundle, we start with a representation $\Bbb V$ of $G$, so the
infinitesimal representation defines a bundle map $\bullet\colon \Cal AM\x\Cal VM\to\Cal
VM$ and we denote the corresponding tensorial map on sections by the same symbol. In
the case that $\Cal VM=\Cal AM$, the infinitesimal representation is just the adjoint
representation of $\frak g$, so for $s,t\in\Ga(\Cal AM)$, we get $s\bullet t=\{s,t\}$
where $\{\, ,\, \}$ is induced by the Lie bracket on $\frak g$. Recall from Section
\ref{2.3} that there is a natural projection $\Pi\colon \Cal AM\to TM$. The induced
operation on sections relates nicely to the picture of $Q$-invariant vector fields
since any such vector field is projectable, and for $s\in\Ga(\Cal AM)$,
$\Pi(s)\in\frak X(M)$ is just the projection of the $Q$-invariant vector field
corresponding to $s$. It turns out that for a tractor bundle $\Cal VM$, $s\in\Ga(\Cal
AM)$ and $\si\in\Ga(\Cal VM)$ the combination $D_s\si+s\bullet\si$ depends only on
$\Pi(s)$ so this descends to an operation $\Ga(TM)\x\Ga(\Cal VM)\to\Ga(\Cal VM)$, see
again \cite[Section 1.5.8]{book}. This turns out to be a linear connection, the
\textit{tractor connection} $\nabla^{\Cal V}$. In particular, we get the defining
equation for the adjoint tractor connection: $\nabla^{\Cal
 A}_{\Pi(s)}t=D_st+\{s,t\}$.

While things are not formulated in this way in \cite{Rel-BGG2} (see the proof of
Theorem~\ref{thm3.1} for details), relative tractor connections can be constructed in
a very similar fashion. Suppose we have a~bundle $\Cal VM$ that is induced by a
representation of $P$, where $Q\subset P\subset G$ as in Section~\ref{2.2}. Then for any
$g\in Q$, the subspace $\frak p\subset\frak g$ is invariant under $\Ad(g)\colon \frak
g\to\frak g$, so it gives rise to a smooth subbundle in $\Cal AM$ that is denoted by
$\Cal A^{\frak p}M$ in \cite{Rel-BGG2}. The infinitesimal representation again
induces a bundle map $\bullet\colon \Cal A^{\frak p}M\x\Cal VM\to\Cal VM$. As before, we
conclude that the operator $\Ga(\Cal A^{\frak p}M)\x\Ga(\Cal VM)\to \Ga(\Cal VM)$
defined by $(s,\si)\mapsto D_s\si+s\bullet\si$ depends only on $\Pi(s)$, but now~$\Pi(s)$ is a section of $\Cal G\x_Q(\frak p/\frak q)=T_\rho M$. Hence, we do not obtain
a full linear connection but only a~\textit{partial connection} $\nabla^{\rho,\Cal V}\colon
\Ga\bigl(T_\rho M\bigr)\x\Ga(\Cal VM)\to\Ga(\Cal VM)$,
which still is linear over smooth
functions in the first argument and satisfies the Leibniz rule in the second
argument.

In the special case of the relative tractor bundles $\Cal V_{i'}M$ constructed in Section~\ref{2.3}, we prove that the relative tractor connection is induced by the adjoint
tractor connection.

\begin{Theorem}\label{thm3.1}
In the setting of Section~$\ref{2.3}$, consider the subbundles $\Cal A^{(i',*)}M$. Then for
${s\in\Ga\bigl(\Cal A^{(0,*)}M\bigr)}$ and $t\in\Ga\bigl(\Cal A^{(i',*)}M\bigr)$ with $i'\leq 0$, we get
\smash{$\nabla^{\Cal A}_{\Pi(s)}t\in\Ga\big(\Cal A^{(i',*)}\big)M$}. Hence, the adjoint tractor
connection restricts to an operation $\Ga\bigl(T_\rho M\bigr)\x\Ga\bigl(T^{i'}_PM\bigr)\to\Ga\bigl(T^{i'}_PM\bigr)$
for any $i'<0$.

Thus, there is an induced partial connection on $\Ga(\Cal V_{i'}M)$ and this
coincides with the relative tractor connection $\nabla^{\rho,\Cal V_{i'}}$ as
obtained in {\rm \cite[{\it Section} 4.3]{Rel-BGG2}}.
\end{Theorem}
\begin{proof}
Consider the defining formula $\nabla^{\Cal A}_{\Pi(s)}t=D_st+\{s,t\}$ for
$s\in\Ga\bigl(\Cal A^{(0,*)}M\bigr)$ and $t\in\Ga\smash{\bigl(\Cal A^{(i',*)}M\bigr)}$. By \cite[Proposition 1.5.8]{book},
the fundamental derivative preserves natural subbundles, so the first
summand in the right-hand side is a section of $\Cal A^{(i',*)}M$. On the other hand,
Proposition \ref{prop2.2} together with $\frak p=\frak g^{(0,*)}$ shows that $\bigl[\frak
 g^{(0,*)},\frak g^{(i',*)}\bigr]\subset \frak g^{(i',*)}$, so the second summand of the
right-hand side of the defining formula is a section of $\Cal A^{(i',*)}M$, too.

Fixing $i'<0$, we have thus obtained a partial connection on the bundle $\Cal
A^{(i',*)}M$ for which the subbundle $\Cal A^{(i'+1,*)}M$ is parallel, so it descends
to a well-defined partial connection on the quotient bundle $\Cal V_{i'}M=\Cal
A^{(i',*)}M/\Cal A^{(i'+1,*)}M$. Thus, it remains to show that this coincides with the
relative tractor connection. The construction of the latter in \cite{Rel-BGG2} was
based on a~variant of the fundamental derivative for relative natural bundles: One
defines the relative adjoint tractor bundle $\Cal A_\rho M=\Cal G\x_Q(\frak p/\frak
p_+)$, so in our notation, this is just $\Cal A^{(0,*)}M/\Cal A^{(1,*)}M$. Given
a~representation $\Bbb W$ of $Q$ on which $P_+$ acts trivially, one can first restrict
the fundamental derivative $D_s\si$ to $s\in\Ga\bigl(\Cal A^{(0,*)}M\bigr)$ and then observe
that $D_s\si=0$ if $s\in\Ga\bigl(\Cal A^{(1,*)}M\bigr)$, to obtain a~well-defined operator
$D^\rho\colon \Ga(\Cal A_\rho M)\x\Ga(\Cal WM)\to\Ga(\Cal WM)$. If $\Cal VM$ is a relative
tractor bundle, then $\Bbb V$ is the restriction of a representation of $P$, so the
infinitesimal representation descends to an operation $\bullet\colon \frak p/\frak
p_+\x\Bbb V\to\Bbb V$. Similarly as above, the relative tractor connection is then
characterized by $\nabla^{\rho,\Cal V}_{\Pi(s)}\si=D^\rho_s \si+s\bullet\si$, where
now $\Pi$ is induced by the canonical projection~${\frak p/\frak p_+\to\frak p/\frak
q}$.

But from this description the claim follows easily: Take sections $\si\in\Ga(\Cal
V_{i'}M)$ and $\xi\in T_\rho M$. For our construction, we choose representative
sections \smash{$t\in\Ga\bigl(\Cal A^{(i',*)}M\bigr)$} for $\si$ and \smash{$s\in\Ga\bigl(\Cal A^{(0,*)}M\bigr)$} for
$\xi$ and project \smash{$D_s t+\{s,t\}\in\Ga\bigl(\Cal A^{(i',*)}M\bigr)$} to the quotient bundle
$\Cal V_{i'}M$. Denoting by $\underline{s}$ the projection of $s$ to $\Cal A_\rho M$
the fact that $\frak p_+$ acts trivially on $\frak g^{(i',*)}/\frak g^{(i'+1,*)}$
implies that the projection of $D_s t$ to $\Ga(\Cal V_{i'}M)$ depends only on
$\underline{s}$ and then by construction it coincides with $D^\rho_{\underline
 s}\si$. Similarly, since $\bigl[\frak g^{(1,*)},\frak g^{(i',*)}\bigr]\subset\frak
g^{(i'+1,*)}$ we conclude that the projection of $\{s,t\}$ to $\Ga(\Cal V_{i'}M)$
depends only on $\underline{s}$ and coincides with $\underline{s}\bullet\si$.
\end{proof}

\begin{Remark}\label{rem3.1}
 The description of a relative tractor connection as induced by the adjoint tractor
 connection leads to explicit descriptions in terms of the structure underlying a
 parabolic geometry of type $(G,Q)$. This works via the concept of \textit{Weyl
 structures} which was introduced in~\cite{Weyl}, see also \cite[Section 5]{book}.
 Any parabolic geometry carries a family of Weyl structures which form
 an affine space modelled on the space $\Om^1(M)$ of one-forms on the manifold
 $M$. Any Weyl structure provides a~linear connection on any natural vector
 bundle called the Weyl connection. Once Weyl connections are understood there is a
 general theory providing explicit formulae for tractor connections, see \cite[Section 5.1]{book}.
\end{Remark}

\subsection{The relation to the Lie bracket}\label{3.2}
As we have noted above, sections of the adjoint tractor bundle $\Cal AM$ can be
identified with $Q$-invariant vector fields on the total space $\Cal G$ of the Cartan
bundle. Now the subspace of $Q$-invariant vector fields is closed under the Lie
bracket of vector fields, whence we obtain an induced bilinear operation on $\Ga(\Cal
AM)$, which we denote by the usual symbol $[\, ,\, ]$. As we have noted in Section
\ref{3.1}, for $s\in\Ga(\Cal AM)$, $\Pi(s)\in\frak X(M)$ is the projection of the
$Q$-invariant vector field corresponding to $s$. Hence, standard properties of the Lie
bracket imply that for $s,t\in\Ga(\Cal AM)$, we get $\Pi([s,t])=[\Pi(s),\Pi(t)]$ with
the Lie bracket of vector fields in the right-hand side.

It is rather easy to express the Lie bracket of sections of $\Cal AM$ in terms of the
operations introduced above and of the curvature $\ka$ of the Cartan connection $\om$,
see \cite[Corollary 1.5.8]{book}: For $s,t\in\Ga(\Cal AM)$, one has
\begin{equation}\label{bracket}
[s,t]=D_st-D_ts-\ka(\Pi(s),\Pi(t))+\{s,t\}.
\end{equation}
Using this and Theorem~\ref{thm3.1}, we can now give an alternative description of
the relative tractor connections on the relative tractor bundles $\Cal
V_{i'}M$. Recall from Section~\ref{2.3} that we have subbundles $TM\supset
T^{i'}_PM\supset T_\rho M$ such that $\Cal V_{-1}M=T^{-1}_P M/T_\rho M$ and $\Cal
V_{i'}M=T^{i'}_PM/T^{i'+1}_PM$ for~${i'<-1}$. Hence, for a smooth section $\si$ of $\Cal
V_{i'}M$ one can always chose a representative vector field in $\Ga\bigl(T^{i'}_PM\bigr)$ that
projects onto $\si$.

\begin{Theorem}\label{thm3.2.1}
In the setting of Section~$\ref{2.3}$, consider vector fields $\xi\in\Ga\bigl(T_\rho M\bigr)$,
$\eta\in\Ga\bigl(T^{i'}_PM\bigr)$ and a section $\si\in\Ga(\Cal V_{i'}M)$. Let
$\tau\in\Om^2(M,TM)$ be the torsion of the Cartan connection $\om$.
\begin{itemize}\itemsep=0pt
\item[$(1)$] The vector field $[\xi,\eta]+\tau(\xi,\eta)$ is a section of the subbundle
$T^{i'}_PM\subset TM$.

\item[$(2)$] If $\eta$ is a representative for $\si$, then the projection of
$[\xi,\eta]+\tau(\xi,\eta)$ to $\Ga(\Cal V_{i'}M)$ depends only on $\si$ $($and not on
the choice of $\eta)$ and coincides with \smash{$\nabla^{\rho,\Cal V_{i'}}_\xi\si$}.
\end{itemize}
\end{Theorem}
\begin{proof}
We can choose smooth representatives $s\in\Ga\bigl(\Cal A^{(0,*)}M\bigr)$ for $\xi$ and
$t\in\Ga\bigl(\Cal A^{(i',*)}M\bigr)$ for $\eta$, so $\Pi(s)=\xi$ and $\Pi(t)=\eta$. As we have
noted above, this implies $\Pi([s,t])=[\xi,\eta]$, so applying $\Pi$ to formula
\eqref{bracket} we obtain
\[
[\xi,\eta]=\Pi(D_st+\{s,t\})+\Pi(D_ts)-\Pi(\ka(\xi,\eta)).
\]
The last summand in the right-hand side by definition is $-\tau(\xi,\eta)$, while
naturality of the fundamental derivative implies that $D_t s\in\Ga\bigl(\Cal A^{(0,*)}M\bigr)$
and hence $\Pi(D_ts)\in\Ga\bigl(T_\rho M\bigr)$. From Theorem~\ref{thm3.1}, we know that the
first term in the right-hand side is \smash{$\Pi\bigl(\nabla^{\Cal AM}_\xi t\bigr)$} and that this is a~section
of $T^{i'}_PM$, so (1) follows. Theorem~\ref{3.1} also tells us that, if
$\eta$ represents $\si$, then projecting \smash{$\Pi\bigl(\nabla^{\Cal AM}_\xi t\bigr)$} further to
$\Ga(\Cal V_{i'}M)$, we obtain \smash{$\nabla^{\rho,\Cal V_{i'}}_\xi\si$}. Since $\Pi(D_ts)$
lies in the kernel of the latter projection, we obtain (2).
\end{proof}

The relation to the Lie bracket is also crucial for obtaining an interpretation of
the relative tractor bundles introduced in Section~\ref{2.3} in the case that the relative
tangent bundle $T_\rho M$ is involutive (and additional assumptions are
satisfied). As observed in Section~\ref{2.3}, if $T_\rho M$ is involutive, then locally
around any $x\in M$, we can find a local leaf space $\psi\colon U\to N$. This means that
$U\subset M$ is open with $x\in U$ and $\ps$ is a surjective submersion onto a smooth
manifold such that for each $y\in U$ the kernel $\ker(T_y\psi)\subset T_y M$
coincides with the fiber of the relative tangent bundle in $y$. In particular, this
implies that we obtain an isomorphism $TU/T_\rho U\cong\psi^* TN$ induced by
$T\psi|_U$. Hence, there is the hope that the subbundles $T^{i'}_PM\subset TM$ that
contain $T_\rho M$ descend to subbundles in $TN$. We can nicely characterize when
this happens in terms of the torsion of the Cartan geometry. Under stronger
assumptions (which are satisfied in important examples), we can nicely characterize
parallel sections for the relative tractor connection.

\begin{Theorem}\label{thm3.2.2}
In the setting of Section~$\ref{2.3}$, assume that $T_\rho M$ is involutive and consider a
local leaf space $\ps\colon U\to N$ with connected fibers. Fix an index $i'$ and consider
the bundles $T^{i'}_PM\subset TM$ and $T^{i'}_PM/T_\rho M\subset TM/T_\rho M$.
\begin{enumerate}\itemsep=0pt
\item[$(1)$] There is a smooth subbundle $T^{i'}N\subset TN$ that corresponds to
$T^{i'}_PM/T_\rho M$ under the isomorphism $TM/T_\rho M\cong\psi^*TN$ induced by
$T\psi$ if and only if the torsion $\tau$ satisfies $\tau\bigl(T_\rho M,T^{i'}_PM\bigr)\subset
T^{i'}_PM$.

\item[$(2)$] If the condition in $(1)$ is satisfied, then $T\psi$ induces an isomorphism from $\Cal
V_{i'}M|_U$ to $\psi^*\bigl(T^{i'}N/T^{i'+1}N\bigr)$. Assuming in addition that $\tau\bigl(T_\rho
M,T^{i'}_PM\bigr)\subset T^{i'+1}_PM\subset T^{i'}_PM$, a section of $\Cal V_{i'}M|_U$ is
parallel for the relative tractor connection $\nabla^{\rho,\Cal V_{i'}}$ if and only if it
is the pullback of a section of $T^{i'}N/T^{i'+1}N$.
\end{enumerate}
\end{Theorem}
\begin{proof}
 (1) We first observe that by Theorem~\ref{thm3.2.1}, the condition on $\tau$ in (1)
 is equivalent to the fact that $[\xi,\eta]\in\Ga\bigl(T^{i'}_PM\bigr)$ for any
 $\xi\in\Ga\bigl(T_\rho M\bigr)$ and $\eta\in\Ga\bigl(T^{i'}_PM\bigr)$. To prove necessity of this
 condition, we assume that $T^{i'}N\subset TN$ is a subbundle with the required
 property. Then for a~local section $\underline{\eta}$ of $T^{i'}N$ we can first
 pull back to a local section of $T^{i'}_PM/T_\rho M$ and then choose a
 representative section $\eta\in\Ga\bigl(T^{i'}_PM\bigr)$ for this. This means that
 $T_y\psi(\eta(y))=\underline{\eta}(\psi(y))$ on the domain of definition of $\eta$,
 so the vector fields $\eta$ and $\underline{\eta}$ are $\psi$-related. Moreover,
 $\xi\in\Ga\bigl(T_\rho M\bigr)$ is of course $\psi$-related to the zero vector field on
 $N$. This shows that $[\xi,\eta]$ has to be $\psi$-related to the zero vector
 field, too, and hence $[\xi,\eta]\in\Ga\bigl(T_\rho M\bigr)$. Starting with a local frame of
 $T^{i'}N$ and adding a local frame of $T_\rho M$, we conclude that there are local
 frames for $T^{i'}_PM$ consisting of vector fields $\eta$ such that
 $[\xi,\eta]\in\Ga\bigl(T_\rho M\bigr)$ for any $\xi\in\Ga\bigl(T_\rho M\bigr)$, which clearly implies
 that the bracket with $\xi$ sends any section of $T^{i'}_PM$ to a section of this
 subbundle.

To prove sufficiency of the condition, put $n=\dim(M)$, $\ell=\text{rank}\bigl(T^{i'}_PM\bigr)$
and $k=\text{rank}\bigl(T_\rho M\bigr)$. We adapt an argument from the proof of the Frobenius
theorem in \cite[Section 3.18]{Michor:topics}. Recall that given a point $x\in U$ we
can find a Frobenius chart $(V,v)$ (for the involutive distribution~$T_\rho M$)
around $x$ such that $V\subset U$. This means that $v(V)\subset\Bbb R^n$ can be
written as $W_1\x W_2$ for connected open subsets in $\Bbb R^k$ and $\Bbb R^{n-k}$,
such that for each $a\in W_2$, the pre-image $v^{-1}(W_1\x\{a\})$ is an integral
submanifold for the distribution $T_\rho M$. Since $\psi$ is a submersion,
$\psi(V)\subset N$ is open and by construction the $\Bbb R^{n-k}$-component of $v$
descends to a map $\psi(V)\to W_2$. This is smooth by the universal property of
surjective submersions and a local diffeomorphism by construction. Shrinking $V$ if
necessary, this becomes a diffeomorphism and hence can be used as a chart map on~$N$.
We implement this by viewing the local coordinates $v^i$ for $i=k+1,\dots, n$ to
be defined both on $V$ and on $\psi(V)$, and then in these coordinates $\psi$ is
given by $\bigl(v^1,\dots,v^n\bigr)\mapsto \bigl(v^{k+1},\dots,v^n\bigr)$.

Let us denote by $\partial_i$ the coordinate vector field
\smash{$\frac{\partial}{\partial v^i}$} for each $i$. Then by construction
$\partial_1,\dots,\partial_k$ form a local frame for $T_\rho M$ and, shrinking $V$ if
necessary, we can extend them by $\eta_{k+1},\dots,\eta_{\ell}\in\Ga\bigl(T^{i'}_PM|_V\bigr)$
to a local frame for $T^{i'}_PM|_V$. Expanding these sections in terms of the
$\partial_i$, we can leave out all summands with $i\leq k$, so we can assume that
there are smooth functions $a_{ij}\colon V\to\Bbb R$ for $1\leq i\leq n-k$ and
$1\leq j\leq\ell-k$ such that
\smash{$\eta_{j+k}=\sum_{i=1}^{n-k}a_{ij}\partial_{k+i}$}. Viewing $(a_{ij})$ as a matrix of
size $(n-k)\x (\ell-k)$, it follows from linear independence of the fields $\eta_j$
that the matrix $(a_{ij}(y))$ has rank $\ell-k$ for each $y\in V$ and hence has
$\ell-k$ linearly independent rows. Specializing to $y=x$, we can assume (permuting
coordinates in $W_2$ if necessary) that the first $\ell-k$ rows of~$(a_{ij}(x))$ are
linearly independent and hence the top $(\ell-k)\x(\ell-k)$ submatrix is
invertible. Again shrinking $V$ if necessary, we can assume that the corresponding
submatrix in $(a_{ij}(y))$ is invertible for any $y\in V$. Hence, we find smooth
functions $b_{ij}\colon V\to\Bbb R$ for $1\leq i,j\leq \ell-k$ such that
$\sum_r a_{ir}(y)b_{rj}(y)=\delta_{ij}$ for any $i$, $j$ and any $y\in V$.

Now we define \smash{$\tilde\eta_{k+j}:=\sum_{i=1}^{\ell-k}b_{ij}\eta_{k+i}$} for
$i=1,\dots,\ell-k$. By construction, these are smooth sections of $T^{i'}_PM$ which in
each point span the same subspace as the $\eta$'s and hence also extend~$\partial_1,\dots,\partial_k$
to a local frame of $T^{i'}_PM$ on $V$. Plugging in the
expansions of the $\eta$'s in terms of the coordinate vector fields, one immediately
concludes that $\tilde\eta_{k+j}=\partial_{k+j}+\sum_{i=\ell+1}^nc_{ij}\partial_i$
for some smooth functions $c_{ij}\colon V\to\Bbb R$. Assuming that
$[\xi,\eta]\in\Ga\bigl(T^{i'}_PM\bigr)$ for any $\xi\in\Ga\bigl(T_\rho M\bigr)$ and
\smash{$\eta\in\Ga\bigl(T^{i'}_PM\bigr)$} we conclude that for each $i=1,\dots,k$,
\smash{$[\partial_i,\tilde\eta_{k+j}]$} can be expanded in as a smooth linear combination of
$\partial_1,\dots,\partial_k,\tilde\eta_{k+1},\dots,\tilde\eta_{\ell}$. But this
bracket is given by $\sum_{r\geq\ell+1}(\partial_i\cdot c_{rj})\partial_r$ and in
view of the form of the $\tilde\eta$'s this is only possible if
$[\partial_i,\tilde\eta_{k+j}]=0$ for any $i=1,\dots k$ and $j=1,\dots,\ell-k$. But
this implies that all the functions $c_{rj}$ satisfy $\partial_i\cdot c_{rj}=0$ for
all $i$, whence they are independent of the coordinates $v^1,\dots,v^k$ and hence
constant along the fibers of $\psi$.

Given a point $x_0\in U$ the image of the fiber of $T^{i'}_PM$ over $x_0$ is a linear
subspace of $T_{\psi(x_0)}N$ of dimension $\ell-k$ and we consider the set
$A\subset \psi^{-1}(\psi(x_0))$ consisting of all points $x$ which lead to the same
subspace. For $x\in A$, the above argument shows that
$V\cap\psi^{-1}(\psi(x))\subset A$, so $A$ is open. Since $A$ is evidently closed,
connectedness of the fibers of $\psi$ implies that $A=\psi^{-1}(\psi(x_0))$ which
completes the proof of (1).

(2) The first statement is obvious by part (1). Under the additional assumption on
$\tau$, part~(2) of Theorem~\ref{thm3.2.1} shows that for a representative
$\eta\in\Ga\bigl(T^{i'}_PM\bigr)$ for $\si\in\Ga(\Cal V_{i'}M)$ one obtains \smash{$\nabla^{\rho,\Cal
 V_{i'}}_\xi\si$} as the projection of $[\xi,\eta]$. Hence, what we have to prove here
is that $\nabla^{\rho,\Cal V_{i'}}\si=0$ if and only if for one or equivalently any
representative $\eta$ for $\nabla$ we get $[\xi,\eta]\in\Ga\bigl(T^{i'+1}_PM\bigr)$ for any
$\xi\in\Ga\bigl(T_\rho M\bigr)$. Now we can run the argument with a Frobenius chart as in the
proof of (1) in two steps. First, we extend the local frame
$(\partial_i)_{i=1,\dots,k}$ for $T_\rho M$ by vector fields
$\eta_{k+1},\dots,\eta_{k+\ell'}$ to a local frame of~$T^{i'+1}_P M$ in such a way
that $\eta_{k+j}=\partial_{k+j}+\sum_{r>k+\ell'}a_r\partial_r$ for some smooth
functions $a_r$. Then we can further extend by
$\tilde\eta_{k+\ell'+1},\dots,\tilde\eta_{k+\ell}$ to a local frame for
$T^{i'}_PM$. As before, we can obviously assume that the $\tilde\eta$'s are smooth
linear combinations of the coordinate vector fields $\partial_r$ with $r>k+\ell'$
only. As in part (1), we can then modify these to fields
$\eta_{k+\ell'+1},\dots,\eta_{k+\ell}$ such that
$\eta_{k+\ell'+j}=\partial_{k+\ell'+j}+\sum_{r>k+\ell}c_{rj}\partial_r$ for smooth
functions $c_{rj}$.

In particular, as in part (1) the vector fields
$\eta_{k+\ell'+1},\dots,\eta_{k+\ell}$ project to a local frame for~$\Cal V_{i'}M$
that descends to a local frame for $T^{i'}N/T^{i'+1}N$. Now in the domain of our
Frobenius chart, take a section of $\Cal V_{i'}M$ and a representative in
$\Ga\bigl(T^{i'}_PM\bigr)$. Expanding this in terms of our local frame, we can drop all
summands except the last $\ell-\ell'$ ones without changing the projection to
$\Ga(\Cal V_{i'}M)$. Hence, we find a representative of the form
\smash{$\eta=\sum_{j=1}^{\ell-\ell'}c_j\eta_{k+\ell'+j}$} and
\smash{$[\partial_i,\eta]=\sum_{j=1}^{\ell-\ell'}(\partial_i\cdot c_j)\eta_{k+\ell+j}$}. The
form of $\eta_{k+\ell+j}$ readily implies that this is a section of $T^{i'+1}M$ if
and only if it is zero and hence to $\partial_i\cdot c_j=0$ for any $i$ and $j$. But
this is equivalent to each $c_j$ being independent of the first $k$
coordinates. Hence, $\eta$ descends to a local section of $T^{i'}N/T^{i'+1}N$ whose
pullback visibly coincides with the original section of $\Cal V_{i'}M$. Conversely,
in our frame, any such pullback has coordinate functions that descend to $N$.
\end{proof}

\subsection{The distinguished class of relative tractor bundles}\label{3.3}
In the setting of Section~\ref{2.3}, let us assume that we have given a parabolic geometry
${(p\colon \Cal G\!\to\! M,\om)}$ whose torsion $\tau$ satisfies $\tau\bigl(T_\rho M,T^{i'}_PM\bigr)\subset
T^{i'}_PM$ for all $i'\leq 0$. Then $T_\rho M$ is involutive and by~Theorem~\ref{thm3.2.2} for any local leaf space $\psi\colon U\to N$ for the induced foliation with
connected fibers, we get a canonical filtration of the tangent bundle by the smooth
subbundles $T^{i'}N$. As usual, one can form the associated graded of this
filtration, which has the form
\begin{gather*}
 \gr(TN)=\bigoplus_{i'}\gr_{i'}(TN) \qquad\text{with}\quad \gr_{i'}(TN)=T^{i'}N/T^{i'+1}N
\quad\text{and}\quad T^0N=N\x \{0\}.
\end{gather*}

Performing natural constructions with vector bundles, one always obtains induced
filtrations of the resulting bundles, and forming the associated graded is compatible
with the constructions. For example, the natural filtration of the cotangent bundle
$T^*N$ has positive indices with the~$j$th filtration component being the annihilator
of $T^{-j+1}N$ so $TN=T^1N\supset T^2N\supset\cdots$. For the associated graded, this
readily implies that $\gr_j(T^*N)$ is dual to $\gr_{-j}(TN)$. For filtered bundles~$\bigl(E,E^i\bigr)$
and $\bigl(F,F^j\bigr)$, the components of the filtration on the tensor product
$E\otimes F$ are defined by letting $(E\otimes F)^\ell$ be the span of the component
$E^i\otimes F^j$ with $i+j=\ell$. For the associated graded, this means that
$\gr_\ell(E\otimes F)=\bigoplus_{i+j=\ell}\gr_i(E)\otimes\gr_j(F)$. See
\cite[Section 3.1.1]{book} for more details.

This allows us to formulate the description of the special class of relative tractor
bundles we are aiming at.
\begin{Corollary}[to Theorems \ref{thm3.2.1} and \ref{thm3.2.2}] \label{cor3.3}
In the setting of Section~$\ref{2.3}$, consider a parabolic geometry $(p\colon \Cal G\to M,\om)$
whose torsion $\tau$ satisfies $\tau\bigl(T_\rho M,T^{i'}_PM\bigr)\subset T^{i'}_PM$ for all
$i'\leq 0$. Let~${\psi\colon U\to N}$ be a local leaf space for the involutive distribution
$T_\rho M=T^0_PM$ with connected fibers.
\begin{itemize}\itemsep=0pt
\item[$(1)$] For any $i'$ we get $\Cal V_{i'}M|_U=\psi^*(\gr_{i'}(TN))$. Applying any
tensorial construction to the bundles $\Cal V_{i'}M$, one obtains a relative tractor
bundle whose restriction to $U$ can be naturally identified with the pullback of the
same construction applied to the bundles $\gr_{i'}(TN)$.
\item[$(2)$] Suppose that the geometry even satisfies $\tau\bigl(T_\rho M,T^{i'}_PM\bigr)\subset
T^{i'+1}_PM$ for all $i'<0$. Then for any relative
tractor bundle constructed as in $(1)$ a local section over $U$ is parallel for the
relative tractor connection if and only if it is the pullback of a section of the
corresponding bundle over~$N$.
\end{itemize}
\end{Corollary}
\begin{proof}
 (1) The first part immediately follows from Theorems~\ref{thm3.2.1} and~\ref{thm3.2.2}. Applying a tensorial construction vector bundles associated to a
 principal bundle, one always obtains the bundle associated to the representation
 obtained from the corresponding construction on the inducing representations. Hence, any tensorial construction with the bundles $\Cal V_{i'}M$ leads to a relative
 tractor bundle. The last part then follows from the compatibility of tensorial
 constructions with the passage from filtered vector bundles to associated graded
 bundles as discussed above and with pullbacks.

(2) Recall the construction of the relative tractor connection in terms of the
 fundamental derivative (or its relative analog) and the infinitesimal
 representation as described in Section~\ref{3.1}. Naturality of the fundamental
 derivative (see \cite[Section 1.5.8]{book}) then implies that applying
 a~tensorial construction to the bundles $\Cal V_{i'}M$, the relative tractor
 connection on the resulting relative tractor bundle coincides with the connection
 induced by the relative tractor connection on the bundles $\Cal V_{i'}M$. Knowing
 this, the general claim follows from Theorem~\ref{thm3.2.2}.
\end{proof}

\subsection{Example: Legendrean contact structures}\label{3.4}
Legendrean contact structures (also called Lagrangean contact structures following
the article~\cite{Takeuchi} in which they were first studied) are an example of
parabolic contact structures, so they can be viewed as a refinement of contact
structures. Explicitly, a Legendrean contact structure on a smooth manifold of odd
dimension $2n+1$ is given by a contact structure $H\subset TM$ together with a
decomposition $H=E\oplus F$ of $H$ as a direct sum of two subbundles of rank $n$,
which are Legendrean in the sense that the Lie bracket of two sections of one of the
subbundles always is a section of $H$. Otherwise put, for any $x\in M$ the fibers
$E_x,F_x\subset H_x\subset T_xM$ are maximal isotropic subspaces with respect to the
non-degenerate bilinear form $\Cal L_x\colon H_x\x H_x\to T_xM/H_x$ induced by the Lie
bracket of vector fields. The interest in these structures mainly comes from their
relation to the geometric theory of differential equations, see, e.g., \cite{MFMZ} and
to CR geometry, see, e.g., \cite{Doubrov-Merker-The}. Indeed, Legendrean contact
structures and partially integrable almost CR structures are two real forms of the
same complex geometric structure.

The relation to parabolic geometries comes from the homogeneous model of such
geometries, which is the partial flag manifold $F_{1,n+1}\bigl(\Bbb R^{n+2}\bigr)$ of lines
contained in hyperplanes in $\mathbb R^{n+2}$. This is a homogeneous space of the
group $G:={\rm PGL}(n+2,\Bbb R)$ and the stabilizer of the standard flag
$\Bbb R\subset\Bbb R^{n+1}\subset\Bbb R^{n+2}$ descends to a parabolic subgroup
$Q\subset G$. Indeed, $Q=P\cap\tilde P$ for two maximal parabolic subgroups of $G$,
the stabilizers of $\Bbb R\subset\Bbb R^{n+2}$ and $\Bbb R^{n+1}\subset\Bbb
R^{n+2}$. Correspondingly, there is a canonical homogeneous projection
$G/Q\to G/P=\Bbb RP^{n+1}$ and it is easy to see that this identifies $G/Q$ with the
projectivized cotangent bundle $\Cal P\bigl(T^*\Bbb RP^{n+1}\bigr)$. This gives rise to a
canonical contact structure on $G/Q$ for which the vertical subbundle is well known
to be Legendrean. The second Legendrean subbundle making $G/Q$ into a Legendrean
contact manifold is the vertical bundle of the canonical projection
$G/Q\to G/\tilde P\cong\Bbb RP^{(n+1)*}$.

On the level of Lie algebras, we also have
$\frak q=\frak p\cap\tilde{\frak p}\subset\frak g=\frak{sl}(n+2,\Bbb R)$, where
$\frak p$ and $\tilde{\frak p}$ are the stabilizers of $\Bbb R\subset\Bbb R^{n+2}$
and $\Bbb R^{n+1}\subset\Bbb R^{n+2}$, respectively. The parabolics $\frak p$ and
$\tilde{\frak p}$ are exchanged by an outer automorphism of $\frak g$, so we only
discuss $\frak q\subset\frak p\subset\frak g$. Decomposing matrices of size~${n+2}$
into blocks of size $1$, $n$, and $1$, the bigrading induced by this pair has the
form
\begin{equation}\label{bigrading}
\begin{pmatrix} \frak g_{(0,0)} & \frak g_{(1,0)} & \frak g_{(1,1)} \\
 \frak g_{(-1,0)} & \frak g_{(0,0)} & \frak g_{(0,1)} \\
 \frak g_{(-1,-1)} & \frak g_{(0,-1)} & \frak g_{(0,0)}
\end{pmatrix}.
\end{equation}
In particular, we just get one representation $\Bbb V_{i'}$ in this case, namely $\Bbb
V_{-1}=\frak g/\frak p$ and the $Q$-invariant filtration of this representation
consists of the single invariant subspace $(\frak g_{(-1,0)}\oplus\frak p)/\frak
p$. The corresponding bundles are easy to interpret, too.

General results on parabolic geometries imply that for any Legendrean contact
structure $(M,H=E\oplus F)$ there is a normal Cartan geometry $(p\colon \Cal G\to M,\om)$
of type $(G,Q)$ which is unique up to isomorphism, see \cite[Section 4.2.3]{book}.
It is a general fact about Cartan geometries that $\Cal G\x_Q(\frak
g/\frak q)\cong TM$ and the characterizing property of the geometry (apart from
normality) is that the $Q$-invariant subspaces $\bigl(\frak g_{(-1,0)}\oplus \frak
q\bigr)/\frak q\subset\frak g/\frak q$ and $\bigl(\frak g_{(0,-1)}\oplus \frak q\bigr)/\frak
q\subset\frak g/\frak q$ correspond to the subbundles $E$ and $F$,
respectively. Since $\frak g_{(0,-1)}\oplus \frak q=\frak p$ we conclude that
$T_{\rho}M=F$ in this case, and that our basic tractor bundle $\Cal V_{-1}M$ is
simply $TM/F$. The natural filtration of this bundle is given by the single invariant
subspace $H/F\subset TM/F$.

Now we can collect our results in the case of Legendrean contact structures. We
restrict to the case $n>1$, i.e., $\dim(M)\geq 5$. The three-dimensional case can be
easily dealt with directly. It is not very interesting from the point of view of
relative BGG theory, however, since relative BGG sequences consist only of a single
operator in this dimension.

In the current situation, relative tractor bundles correspond to representations of
the general linear group ${\rm GL}(n+1,\Bbb R)$, so they correspond to natural vector
bundles on smooth manifolds of dimension $n+1$. The relative tractor bundle
$\Cal VM=TM/F$ corresponds to the standard representation of ${\rm GL}(n+1,\Bbb R)$ on
$\Bbb R^{n+1}$, so via tensorial constructions as in Section~\ref{3.3} we obtain all
analogs of tensor bundles. It is clear that the latter are by far the most important
natural vector bundles (in fact, we are not aware of interesting applications of
other types of natural bundles). Thus, it should be clear that we obtain (up to
isomorphism) all the important relative tractor bundles on Legendrean contact
structures via the constructions of Section~\ref{3.3}, but it is difficult to make a
precise statement in this direction. Instead of trying to do this, we just prove a
result on the resulting representations in the following theorem, which shows the
richness of the class.

\begin{Theorem}\label{thm3.4}
 Let $(M,H=E\oplus F)$ be a Legendrean contact structure of dimension ${2n+1\geq 5}$
 such that the distribution $F\subset TM$ is involutive and consider the relative
 tractor bundle $\Cal VM:=TM/F$. Applying tensorial constructions to $\Cal VM$ as in
 Section~$\ref{3.3}$, we arrive at a class of relative tractor bundles that correspond to
 representations of $P/P_+\cong {\rm GL}(n+1,\Bbb R)$ which contains all one-dimensional
 representations and whose restrictions to
 $\frak{sl}(n+1,\Bbb R)\subset\frak p/\frak p_+$ exhaust all finite-dimensional
 representations of this Lie algebra.

 For any of these relative tractor bundles and any local leaf space $\psi\colon U\to N$
 for the distribution~$F$ with connected fibers, the induced relative BGG sequence
 defines a fine resolution of the sheaf of pullbacks of sections of a tensor bundle
 on $N$. This tensor bundle is obtained by applying the tensorial constructions from
 above to the tangent bundle $TN$.
\end{Theorem}
\begin{proof}
 We first need some information on the curvature of Legendrean contact structures
 that is available in the literature. For $n>1$, \cite[Proposition 4.3.2]{book}
 completely describes the harmonic curvature of Legendrean contact structures: There
 are two harmonic curvature quantities of homogeneity one, which are sections of the
 bundles $\La^2E^*\otimes F$ and $\La^2F^*\otimes E$, respectively. These turn out
 to be induced by the Lie bracket of vector fields and they are exactly the
 obstructions to involutivity of the distributions $E$ and $F$. So in particular,
 one of these components is exactly the obstruction to involutivity of $T_\rho
 M=F$. The third and last harmonic curvature quantity has homogeneity two and can be
 interpreted as a section of $E^*\otimes F^*\otimes L(E,E)$ (with some symmetries).

Hence, involutivity of $F$ implies that the harmonic curvature vanishes upon insertion
of two elements of $F$ while the components of torsion type even vanish upon
insertion of one element of~$F$. These are exactly the assumptions needed to apply
part (1) of \cite[Proposition~4.18]{Rel-BGG2}. The~proof of this proposition then
shows that the full Cartan curvature has the same properties, which in our language
means that $\tau\bigl(T_\rho M,TM\bigr)=0$, so the assumptions of Corollary~\ref{cor3.3} are
satisfied. On the other hand, \cite[Proposition~4.18]{Rel-BGG2} says that the
assumptions of part (1) of~\mbox{\cite[Theorem 4.11]{Rel-BGG2}} are satisfied. This shows
that any relative BGG sequence is a fine resolution of the sheaf of local parallel
sections for the relative tractor connection. Together with Corollary \ref{cor3.3},
this implies all claimed properties of the relative BGG sequence for $\Cal
VM=TM/F$. For bundles obtained via tensorial constructions, they then follow readily
since relative tractor connections are compatible with tensorial constructions.

So it remains to verify the claim about the representations corresponding to our
relative tractor bundles, i.e., those obtained by tensorial constructions from the
standard representation. Putting
$\frak s:=\frak{sl}(n+1,\Bbb R)\subset\frak p/\frak p_+$ semi-simplicity of $\frak s$
implies that it acts trivially on any one-dimensional representation of
$P/P_+$. Hence, also the subgroup ${\rm SL}(n+1,\Bbb R)\subset P/P_+$ acts trivially so any
such representation factorizes through the determinant representation
$\det\colon {\rm GL}(n+1,\Bbb R)\to\Bbb R\setminus\{0\}$. The determinant representation is
realized via the top exterior power of the standard representation. The square of
this line bundle corresponds to a~representation with values in $\Bbb R^+$, so we
can form powers of that line bundle with arbitrary non-zero real
exponents. These together with their tensor products with $\La^nTM$ exhaust all
possible one-dimensional representations of $P/P_+$.

On the other hand, it is well known that any irreducible representation of $\frak s$
can be realized in an iterated tensor product of copies of the standard
representation and its dual via Young symmetrizers. By construction, any Young
symmetrizer is $P/P_+$-equivariant, and hence each of these irreducible components is
$P/P_+$-invariant. Thus, we see that any irreducible representation of $\frak s$ is
realized by a tensorial construction and since forming direct sums is no problem, we
can arrive at an isomorphic copy of any finite-dimensional representation of $\frak
s$ from a tensorial construction.
\end{proof}

\begin{Remark}\label{rem3.4} \quad
\begin{enumerate}\itemsep=0pt
\item[(1)] It is clear that there are numerous examples of
Legendrean contact structures for which $F$~is involutive but the harmonic curvature
component that is a section of the bundle $E^*\otimes F^*\otimes L(E,E)$ is non-zero.
This follows, for example, from the relation of Legendrean contact structures to
differential equations as used in \cite{MFMZ}. In such cases, Theorem~\ref{thm3.4}
applies to give a~precise description of the sheaves resolved by relative BGG
sequences. However, these geometries do not fall in the class of correspondence
spaces as treated in part (2) of~\mbox{\cite[Theorem~4.11]{Rel-BGG2}}. Hence, in these
cases no explicit description of the resolved sheaves was known before.
\item[(2)] The idea to use quotients of $TM$ as basic relative tractor bundles originally
arose in the context of Legendrean contact structures during the work on the thesis
\cite{Michal-thesis} of the third author under the direction of the first
author. Legendrean contact structures play only a~minor role as an example in the
thesis, however. The thesis mainly explores the fact that the machinery of relative
BGG sequences extends beyond the case of Legendrean contact structures, namely to
contact manifolds that are endowed with either a single involutive Legendrean
distribution or with a Legendrean distribution with an arbitrary chosen complement in
the contact distribution. In any case, the quotient of $TM$ by the distinguished
Legendrean distribution plays a~pivotal role in these developments, since this can be
shown to carry a canonical partial linear connection obtained from the Bott
connection as in Section~\ref{3.2} or from a family of distinguished partial linear
connections determined by a choice of contact form, compare to \cite[Section 5.2.14]{book}
and to~\cite{Michal-article}. The latter are restrictions of Weyl
connections, so one obtains descriptions of relative tractor bundles and relative
tractor connections in terms of Weyl structures as discussed in~Remark~\ref{rem3.1}.
\end{enumerate}
\end{Remark}

\subsection{Example: Generalized path geometries}\label{3.5}
These geometries represent a kind of a dual situation to the case of Legendrean
contact structures discussed in Section~\ref{3.4} above, in the sense that locally they
have a projectivized tangent bundle as an underlying structure. They can be defined
directly by a configuration of distributions as follows. On a smooth manifold $M$ of
dimension $2n+1$, one requires a distribution $H\subset TM$ of rank $n+1$ together
with a decomposition $H=E\oplus V$ where $E$ and $V$ have rank $1$ and $n$,
respectively, such that the bundle map $H\otimes H\to TM/H$ induced by the Lie
bracket of vector fields vanishes on $V\otimes V$ and restricts to an isomorphism on
$E\otimes V$.

It then turns out that for $n\neq 2$, the distribution $V$ is involutive and locally
$M$ is diffeomorphic to an open subset in the projectivized tangent bundle of a local
leaf space for $V$. For $n=2$, we will assume that this is the case, since this is
necessary in order to obtain relative BGG resolutions. Assuming that one deals with
a global projectivized tangent bundle, $M=\Cal PTN$, the subbundle $E\subset TM$
defines a one-dimensional foliation on $M$ and projecting the leaves to~$N$, one
obtains a family of paths (unparametrized curves) on $N$ with exactly one path
through each point in each direction. This is a classical path geometry and
generalized path geometries can be considered as a local analog of that (where, for
example, paths can be defined for an open set of directions only). The importance of
these structures comes from the fact that they provide a~way to geometrize systems of
second order ODE, which is the starting point for their study~in~\cite{Fels}.\looseness=1

Similarly as in Section~\ref{3.4}, the relation to parabolic geometries comes from the
homogeneous model, which this time is the flag manifold $F_{1,2}\bigl(\Bbb R^{n+2}\bigr)$ of
lines in planes in $\Bbb R^{n+2}$. So we again put $G:={\rm PGL}\bigl(n+2,\Bbb R\bigr)$ and the
stabilizer $Q$ of $\Bbb R\subset\Bbb R^2\subset\Bbb R^{n+2}$ has the form
$P\cap\tilde P$, where $G/P\cong\Bbb RP^{n+1}$ and $G/\tilde P$ is the Grassmannian
${\rm Gr}\bigl(2,\Bbb R^{n+2}\bigr)$. The distributions $V$ and $E$ are the vertical subbundles of
the projections to $G/P$ and $G/\tilde P$, respectively, and the integral
submanifolds of $E$ project exactly to the projective lines in~$G/P$.

On the level of Lie algebras, the situation is closely parallel to Section~\ref{3.4}, in
particular the bigrading induced by $\frak q\subset\frak p\subset\frak g$ has the
same form as in \eqref{bigrading}, but now with blocks of size~$1$,~$1$, and~$n$, so
now $\frak g_{(-1,0)}$ is one-dimensional, while $\frak g_{(0,-1)}$ and
$\frak g_{(-1,-1)}$ both have dimension~$n$. In~this case, the situation is less
symmetric and we will make some remarks on the ``other direction'' $\frak
q\subset\tilde{\frak p}\subset\frak g$ below. We will also restrict to the case $n>1$
here, for $n=1$ there is no difference between generalized path geometries and
Legendrean contact structures.

The canonical Cartan geometry $(\Cal G\to M,\om)$ associated to $(M,H=E\oplus V)$
here has the property that the $Q$-invariant subspaces
$\bigl(\frak g_{(-1,0)}\oplus\frak q\bigr)/\frak q$ and
$\bigl(\frak g_{(0,-1)}\oplus\frak q\bigr)/\frak q$ of $\frak g/\frak q$ induce the subbundles
$E$ and $V$ of $TM$, respectively. As in Section~\ref{3.4}, one concludes that $T_\rho M=V$
and that the basic relative tractor bundle in this case is $\Cal VM=TM/V$ with the
natural filtration given by $H/V\subset TM/V$. Collecting our results in this case,
we obtain.\looseness=1

\begin{Theorem}\label{thm3.5}
 Let $(M,H=E\oplus V)$ be a generalized path geometry of dimension $2n+1\geq 5$,
 where~$V$ is assumed to be involutive if $n=2$ and consider the relative tractor
 bundle $\Cal VM:=TM/V$. Applying tensorial constructions to $\Cal VM$ as in Section
 $\ref{3.3}$, we arrive at a class of relative tractor bundles that correspond to
 representations of $P/P_+\cong {\rm GL}(n+1,\Bbb R)$ which contain all one-dimensional
 representations and whose restrictions to
 $\frak{sl}(n+1,\Bbb R)\subset\frak p/\frak p_+$ exhaust all finite-dimensional
 representations of this Lie algebra.

 For any of these relative tractor bundles and any local leaf space $\psi\colon U\to N$ for
 the distribution~$F$ with connected fibers, the induced relative BGG sequence
 defines a fine resolution of the sheaf of pullbacks of sections of a tensor bundle
 on $N$. This tensor bundle is obtained by applying the tensorial constructions from
 above to the tangent bundle $TN$.
\end{Theorem}
\begin{proof}
 According to \cite[Section 4.4.3]{book}, there are only two harmonic curvature
 components for generalized path geometries with involutive $V$. One of these is a
 torsion of homogeneity $2$ that is a section of $E^*\otimes (TM/H)^*\otimes V$. The
 other one is a curvature of homogeneity $3$ that is a section of $V^*\otimes
 (TM/H)^*\otimes L(V,V)$ (with additional symmetry properties). Using this
 information, the proof is completed exactly as for Theorem~\ref{thm3.4}.
\end{proof}

\begin{Remark}\label{rem3.5}\quad
\begin{enumerate}\itemsep=0pt
\item[(1)] Similarly as in the case of Legendrean contact structures, Theorem~\ref{thm3.5}
 covers many cases in which no good description of the sheaves resolved by relative
 BGG sequences was known. Indeed, such a description was only known in the case of a
 correspondence space treated in part (2) of \cite[Theorem 4.11]{Rel-BGG2}. In the
 case of generalized path geometries (with involutive subbundle $V$) these are
 exactly those geometries for which the harmonic curvature component of homogeneity
 $3$ mentioned in the proof of Theorem~\ref{thm3.5} vanishes identically. This means
 that locally the family of paths can be realized as geodesics of a torsion-free
 linear connection on the tangent bundle of a local leaf space. Since the
 unparametrized geodesics depend only on the projective equivalence class of the
 connection, such correspondence spaces in dimension $2n+1$ are best described as
 induced from a projective structure in dimension $n+1$. Of course, there are many
 systems of second order ODE which do not admit a description via geodesics, and
 these lead to generalized path geometries for which the harmonic curvature
 component of homogeneity~$3$ is non-trivial.
\item[(2)] As remarked already, the situation is not as symmetric here as for Legendrean
 contact structures. However, from the point of view of this article, the pair
 $\frak q\subset\tilde{\frak p}\subset\frak g$ is much less interesting. This leads
 to $T_\rho M=E$ and the basic relative tractor bundle $TM/E$. On the one hand, the
 fact that $T_\rho M$ has rank one, then implies that all relative BGG sequences
 consist of a single operator only. On the other hand, the class of relative tractor
 bundles provided by our construction is less rich then in the cases discussed in
 Theorems~\ref{thm3.4} and~\ref{thm3.5}. The point here is that essentially $\tilde
 P/\tilde P_+\cong S({\rm GL}(2,\Bbb R)\x {\rm GL}(n,\Bbb R))\subset {\rm GL}(2n,\Bbb R)$ in our case
 and the bundle $TM/E$ corresponds to the tensor product of the dual of the standard
 representation of the ${\rm GL}(2,\Bbb R)$ factor with the standard representation of the
 ${\rm GL}(n,\Bbb R)$-factor. The latter give rise to relative tractor bundles of rank $2$
 respectively $n$, that cannot be obtained by tensorial constructions form~$TM/E$.
 The tensor product decomposition does not descend to a local leaf space for~$E$
 unless the initial geometry is torsion-free, in which case it induces a
 $(2,n)$-Grassmannian structure on the local leaf space.
\item[(3)] The relation of relative tractor bundles to $TM/V$ in the case of generalized
 path geometries was observed during the work on the thesis \cite{Zhangwen-thesis}
 of the second author under the direction of the first author. Relative tractor
 bundles and the relative BGG machinery play only a~minor role in the thesis,
 however. One of the results of the thesis is a complete description of Weyl
 structures for generalized path geometries and relative tractor connections play an
 important role in the characterization of Weyl connections and in the description
 of tractor calculus in terms of Weyl connections. In particular, this also leads to
 a description of relative tractor bundles and relative tractor connections as
 discussed in Remark~\ref{rem3.1}.
\end{enumerate}
\end{Remark}

\subsection{Explicit examples of relative BGG resolutions}\label{3.6}
 Here we give a few explicit examples of relative BGG resolutions for the examples
 treated in the last two sections. Both cases are closely related to standard BGG
 sequences associated to projective structures in appropriate dimension. We do not go
 into the details of how to compute the relevant relative Lie algebra homology
 groups but refer to \cite{Rel-BGG1} and \cite{Rel-BGG2}. We focus on the case of
 generalized path geometries in dimension $7$. This is quite close to the
 five-dimensional case discussed in \cite{Rel-BGG2} but shows a bit more varied
 behavior in simple cases. In the end, we briefly comment on the case of Legendrean
 contact structures.

 We will use the Dynkin diagram notation for irreducible representations, see
\cite[Section 3.2]{book} for general information. In our case, \smash{$\xbbb{}{}{}{}$}
 indicates the parabolic subalgebra $\frak p$ and irreducible representations of
 $\frak p$ (and of the group $P$) are indicated by putting over the nodes the
 coefficients in the expansion of the highest weight in terms of fundamental
 weights. In particular the last three of these numbers have to be non-negative
 integers. Similarly, \smash{$\xxbb{}{}{}{}$} denotes the parabolic subalgebra $\frak q$ and
 with numbers over the nodes irreducible representations of $\frak q$ and of the
 corresponding group $Q$. Here only the last two coefficients have to be
 non-negative integers. Determining the forms of the relative BGG sequences is a
 simple extension of part~(1) in~\mbox{\cite[Example~3.2]{Rel-BGG1}}, and we just state
 the results below. They have the same form as for standard BGG sequences for
three-dimensional projective structures. In the picture of weights, this corresponds
 to dropping the leftmost node of all diagrams.

 So take a generalized path geometry $(M,E\oplus V)$ with $\dim(M)=7$, so $V$ has
 rank $3$. The fact that the basic relative tractor bundle $\Cal VM=TM/V$ is induced
 by the $P$-irreducible quotient of the adjoint representation shows that it
 corresponds to the representation \smash{$\xbbb{1}{0}{0}{1}$}. The~irreducible subbundle
 $(E\oplus V)/V\cong E$ corresponds to \smash{$\xxbb{-2}{1}{0}{0}$} and the quotient of~$TM/V$
 by this subbundle is isomorphic to $TM/(E\oplus V)$ and hence corresponds to
 \smash{$\xxbb{1}{0}{0}{1}$}. Dropping the leftmost nodes in the diagrams exhibits that this
 is analogous to the standard tractor bundle and its composition series for
three-dimensional projective structures.

 One family of well-known BGG sequences in projective geometry is induced by the
 symmetric powers of the standard cotractor bundle, which is dual to the standard
 tractor bundle. Now $(TM/V)^*$ corresponds to \smash{$\xbbb{-2}{1}{0}{0}$} and hence
 $S^k(TM/V)^*$ corresponds to \smash{$\xbbb{-2k}{k}{0}{0}$}. The relative BGG sequence
 induced by each of these bundles starts with the $Q$-irreducible quotient of the
 bundle, which is the density bundle \smash{$\xxbb{-2k}{k}{0}{0}$}. The next bundle in the~relative
 BGG sequence is \smash{$\wxxbb{-k+1}{-k-2}{k+1}0$}. It is easy to see that this is
 obtained by twisting $S^{k+1}V^*$ by an appropriate density bundle. In particular,
 the first operator in the relative BGG sequences has order $k+1$ in this case, and
 by Theorem~\ref{thm3.5} locally its kernel consists of all pullbacks of symmetric
 $\binom{0}{k}$-tensor fields on a local leaf space. One can also deduce from this,
 that all further operators in the relative BGG sequence have order one.

 The full relative BGG sequence for $k=1$, i.e., the bundle $(TM/V)^*$, has the form
 \begin{equation}\label{sequence}
\xxbb{-2}100\to \xxbb0{-3}20 \to\xxbb1{-4}11\to \xxbb2{-5}10.
\end{equation}
Up to a twist by a line bundle, the last two bundles can be interpreted as the
kernel of the complete symmetrization in $V^*\otimes S^2V^*$ and as $V^*$,
respectively.

Another well-known projective BGG sequence is a projectively invariant version of the
Riemannian deformation sequence. For flat metrics it is also known as the Calabi
complex or as the fundamental complex of linear elasticity theory, see
\cite{MikE}. In our setting, this corresponds to the relative BGG sequence induced by
the relative tractor bundle $\La^2(TM/V)^*$, which corresponds to the representation
\smash{$\xbbb{-3}010$}. One then computes that the full relative BGG sequence has the form
\[
\xxbb{-3}010 \to \xxbb{-2}{-2}20 \to \xxbb0{-4}02 \to \xxbb1{-3}01.
\]
Up to a twist by a density bundle, the first two bundles in the sequence are
induced by the representations $V^*$ and $S^2V^*$, so the first operator is like a
symmetrized covariant derivative in $V$-directions. By Theorem~\ref{thm3.5}, locally
its kernel is isomorphic to the pullbacks of $2$-forms on a~local leaf space. Up to
a twist by a density bundle, the third bundle is induced by $S^2V$, and the second
operator is of order two. This operator is the relative analog of the projectively
invariant version of the operator that computes the infinitesimal change of curvature
caused by an infinitesimal deformation of a Riemannian metric. The last bundle is (up
to a twist) induced by $V$ and the last operator roughly is a divergence in
$V$-directions. To make the operators in all the BGG sequences more explicit, one has
to use Weyl structures for path geometries, see \cite{Zhangwen-thesis}.

The relative BGG sequences for Legendrean contact structures as discussed in Section
\ref{3.4} are also similar in form to standard BGG sequences for projective
structures. If the original manifold $M$ has dimension $2n+1$, then relative BGG
sequences are related to BGG sequences for projective structures in dimension
$n$. Here the basic relative tractor bundle $TM/F$ contains the distinguished
subbundle $H/F\cong E\subset TM/F$ of rank $n$, and is corresponds to the dual of the
standard tractor bundle in projective geometry. So in particular, the relative BGG
sequence determined by this bundle starts with a second order operator defined on a
line bundle, and then extends similar to \eqref{sequence} (up to twists by
appropriate line bundles) with $S^2E^*$, the kernel of the complete symmetrization in
$E^*\otimes S^2E^*$, and $E^*$. The operators in the sequence have orders $2$, $1$,
and $1$. Again, they can be made more explicit in terms of (families of)
distinguished connections, see \cite{Michal-thesis}.

\subsection*{Acknowledgements}

For open access purposes, the authors have applied a CC
BY public copyright license to any author-accepted manuscript version arising from
this submission. The first and second authors supported by the Austrian Science Fund (FWF), \href{https://doi.org/10.55776/P33559}{doi:10.55776/P33559}. The second and third authors partially supported by the Austrian Science
Fund (FWF), \href{https://doi.org/10.55776/Y963}{doi:10.55776/Y963}. This article is based upon work from COST Action CaLISTA CA21109
supported by COST (European Cooperation in Science and
Technology, \url{https://www.cost.eu}).
We thank the anonymous referees for several helpful comments and suggestions.

\pdfbookmark[1]{References}{ref}
\LastPageEnding

\end{document}